\documentclass[a4paper,11 pt]{amsart}

\usepackage{amssymb}  %,amscd,amsthm,mathrsfs}
\usepackage{mathtools}      % improves amsmath
\usepackage{mathabx}        % makes many symbols (as $\leq$) more beautiful; must be loaded after mathtools 
\usepackage[bb=fourier,cal=euler,scr=rsfs]{mathalfa}	% selecting fancy math fonts
\usepackage{enumitem}       % special itemize with custom tags and labels that work 
\usepackage[colorlinks=true,backref=page]{hyperref} % Links in the pdf. Backref option: each bibliographical entry denotes where it was cited

\usepackage[ansinew]{inputenc}
\renewcommand*{\backref}[1]{}
\renewcommand*{\backrefalt}[4]{\quad \tiny
  \ifcase #1 (\textbf{NOT CITED.})%
  \or    (Cited on page~#2.)%
  \else   (Cited on pages~#2.)%
  \fi}

\hypersetup{bookmarksdepth = 3} % Depth of sections/subs... to have bookmark links in the pdf;
\numberwithin{equation}{section}     % Makes labeled equations easier to find.

\setlist[enumerate,1]{label={\upshape(\roman*)},ref=\roman*}
\setlist[enumerate,2]{label={\upshape(\alph*)},ref=\alph*}

 \def\NN{{\mathbb N}}  \def\PP{{\mathbb P}}
\def\QQ{{\mathbb Q}} \def\RR{{\mathbb R}}  \def\TT{{\mathbb T}}
   
 \def\ZZ{{\mathbb Z}}

\def\cA{\mathcal{A}}    
    
    \def\cU{\mathcal{U}}

   \def\cR{\mathcal{R}}

\newtheorem*{teo*}{Theorem}

\newtheorem{teo}{Theorem}[section]

\newtheorem{quest}{Question}
\newtheorem{cor}[teo]{Corollary}

\newtheorem{prop}[teo]{Proposition}

\theoremstyle{definition}
\newtheorem{defi}{Definition}
\theoremstyle{remark}

\newtheorem{remark}[teo]{Remark}

\newcommand{\eps}{\varepsilon}

\DeclareMathOperator{\Diff}{Diff}

\title[Uniform expansion]{A remark on uniform expansion}

\author[R.~Potrie]{Rafael Potrie} 
\address{Centro de Matem\'atica, Universidad de la Rep\'ublica, Uruguay}
\email{rpotrie@cmat.edu.uy}
\urladdr{http://www.cmat.edu.uy/~rpotrie/}

\thanks{Rafael Potrie was partially supported by CSIC 618, FCE-1-2017-1-135352. This work was started while the author was a Von Neumann fellow at IAS, funded by Minerva Research Fundation Membership Fund and NSF DMS-1638352.}

\begin{document}

\begin{abstract}
For every $\cU \subset \Diff^\infty_{vol}(\TT^2)$ there is a measure of finite support contained in $\cU$ which is uniformly expanding. 

%\medskip \medskip
%\emph{This is a preliminary version. Please do not distribute.} 
\end{abstract}

\maketitle

Let $\mu$ be a probability measure in $\Diff^r(M)$ where $M$ is a closed manifold of dimension $d := \dim (M)$. We denote by $\mu^{(1)}=\mu$ and $\mu^{(n)} = \mu \ast \mu^{(n-1)}$. Note that $\mu^{(n)}$ is the pushforward by the composition of the product measure $\mu^n$ in $(\Diff^{r}(M))^n$.

\begin{defi}[\cite{EL,BEF}]\label{def-Nexp} A measure $\mu$ in $\Diff^r(M)$ is \emph{uniformly expanding} if there exists $N>0$ so that for every $(x,v) \in T^1M$ one has that: 
$$ \int \log \|D_x f v \| d\mu^{(N)}(f) > 2. $$
\end{defi}

This is a robust\footnote{To be precise, if $\mu$ has compact support, then there is a neighborhood $\cU$ of its support such that for any measure $\mu'$ which has support in $\cU$ and is weak-$\ast$-close to $\mu$, then $\mu'$ is also uniformly expanding (see \eqref{eq:mun2} below).} condition on $\mu$. This notion as well as similar ones have been studied extensively recently as it allows to describe quite precisely the stationary measures for random walks with $\mu$ as law (see \cite{DK,BRH,LX,EL,Chung,BEF}). 

Here we will make one remark (which can be related to some results, e.g. in \cite{BXY,Chung,Zhang}) that points in the direction of the abundance of uniform expansion. 

\begin{teo}\label{teo-main}
For every $\cU$ open set in $\Diff_{vol}(\TT^2)$ there is a finitely supported measure $\mu$ whose support is contained in $\cU$ and $\mu$ is uniformly expanding. 
\end{teo}

As a consequence of the results of \cite{BRH,LX,Chung} one deduces that: 

\begin{cor}\label{cor-main}
For every $\cU \subset \Diff_{vol}(\TT^2)$ there is a measure $\mu$ finitely supported in $\cU$ such that orbit of every point under the random walk on $\TT^2$ produced by $\mu$ equidistributes in $\TT^2$. Moreover, for every $\mu'$ close to $\mu$ in the weak-$\ast$-topology every orbit of the random walk is either finite or dense (in particular, the elements of the support of $\mu$ generate a stably ergodic semigroup\footnote{A semigroup generated by diffeomorphisms $f_1,\ldots, f_k$ is \emph{stably ergodic} if there are neighborhoods $\cU_i$ of $f_i$ such that for every family $\{g_i\}_i$ with $g_i \in \cU_i$ we have that the semigroup generated by $\{g_1, \ldots, g_k\}$ verifies that every set which is invariant under all the $g_i$ has full or zero measure. The argument also gives robust transitivity which also follows by a stronger result \cite{KN}, however, our argument provides robust transitivity even outside the set of volume preserving diffeomorphisms, see Remark \ref{r.nonvolumepres}.}). \end{cor} 

We discuss some other consequences as well as some possible extensions in Section \ref{s.comments}. 

%%%%%%%%%%%%%%%%%%%%%%%%%%%%
\section{Criteria for uniform expansion}

For a measure $\mu$ in $\Diff^r(M)$ one can define a \emph{random walk} on $\Diff^r(M)$ formally defined as follows: consider $\Omega = (\Diff^r(M))^\NN$ with the measure $\mu^\NN$ which is invariant under the shift map $T: \Omega \to \Omega$ that sends a sequence $\omega = (f_0, \ldots, f_n, \ldots )$ to $T\omega = (f_1, \ldots, f_n, \ldots).$ We denote $f_\omega^n = f_{n-1} \circ \ldots \circ f_0$.  Notice that $f_\omega^n$ distributes as $\mu^{(n)}$ if one chooses $\omega \in \Omega$ with law $\mu^\NN$.

Recall that a measure $\nu$ in $M$ is called $\mu$-\emph{stationary} if $\mu \ast \nu = \nu$, equivalently, the measure $\mu^\NN \times \nu$ is invariant under the skew-product dynamics $(\omega, x) \mapsto (T\omega, f_0(x))$. It is ergodic if every measurable set $A \subset M$ which is $f$-invariant for $\mu$-a.e. $f$ verifies that $\nu(A)$ is either $0$ or $1$. %If $\mu$ is finitely supported it is easy to see that $\nu$ has to be quasi-invariant by $\mu$-almost every $f$ (that is, $f_\ast \nu$ is absolutely continuous with respect to $\nu$). 

For an ergodic $\mu$-stationary measure $\nu$ (under some integrability conditions\footnote{For simplicity we will assume that every $\mu$ is boundedly supported.} on $\mu$) there are numbers $\lambda_1(\nu) \geq \ldots \geq \lambda_d(\nu)$ called the \emph{Lyapunov exponents} that are characterized by the fact that for $\nu$-a.e. $x \in M$ and $\mu^\NN$-a.e. $\omega \in \Omega$ there exists some basis $\{v_i\}_i$ of  $T_xM$ so that:

$$ \lim_n \frac{1}{n} \log \|D_x f_\omega^n v_i \| = \lambda_i(\nu). $$

We say that an ergodic $\mu$-stationary measure $\nu$ is \emph{expanding} if $\lambda_1(\nu)>0$. 

Given an expanding $\mu$-stationary measure $\nu$, we say that it has a \emph{non-random weak stable} direction if there exists a measurable subbundle $E \subset TM$ defined for $\nu$-a.e. $x \in M$ so that: 

\begin{itemize}
\item $D_x f(E(x)) = E(f(x))$ for $\mu$-a.e. $f$ and $\nu$-a.e. $x$.  
\item If $v \in E(x)\setminus \{0\}$ then for $\mu^\NN$-a.e. $\omega$ one has that $ \lim_n \frac{1}{n} \log \|D_x f_\omega^n v\| \leq 0$.  
\end{itemize}

The following result is proven \cite[Theorem C]{Chung} (see also \cite{LX}). 

\begin{teo}\label{teo.equivunifexp}
The measure $\mu$ is uniformly expanding if and only if every ergodic $\mu$-stationary measure $\nu$ is expanding and does not admit a non-random weak stable direction. 
\end{teo} 

Notice that in the case where $\mu$ is supported in the space of volume preserving diffeomorphisms of surfaces, every expanding $\mu$-stationary measure must be hyperbolic with one positive and one negative exponent. The only possible non-random weak stable direction is the stable one which is one-dimensional. We refer the reader to \cite{Kifer} for more information on stationary measures. 

%%%%%%%%%%%%%%%%%%%%%%%%%%%%%%%
We state the following criteria, sometimes called the \emph{invariance principle} (see e.g. \cite[Theorem B]{AV}), that we will use to ensure that the measure we are considering is expanding. 

\begin{teo}\label{teo-invpple}
Let $\nu$ be an ergodic $\mu$-stationary measure which is not expanding. Then, $\nu$ is $f$-invariant for $\mu$-a.e. $f$. 
Moreover, if $\nu$ is $f$-invariant for $\mu$-a.e. $f$ and all exponents equal $0$, then there is an invariant $\nu$-measurable distribution or conformal structure. 
\end{teo}

Here, an \emph{invariant} $\nu$-\emph{measurable distribution} means a measurable section $E: M \to \mathrm{Gr}_i(TM)$ (where $\mathrm{Gr}_i(TM)$ denotes the Grasmannian bundle of $i$-planes of $M$) which is well defined modulo sets of $\nu$-measure zero and such that $D_xf(E(x)) = E(f(x))$ for $\mu$-almost every $f$ and $\nu$-almost every $x$. Similarly, an  \emph{invariant} $\nu$-\emph{measurable conformal structure} means a measurable section $E: M \to \mathrm{CS}(TM)$ (where $\mathrm{CS}(TM)$ denotes the bundle of conformal structures over $M$, that is, at each $x \in M$ the fiber $\mathrm{CS}(T_xM)$ corresponds to the space of inner products in $T_xM$ up to homothety\footnote{Or equivalently, if one fixes a Riemannian metric on $M$, we can identify $\mathrm{CS}(T_xM)$ with the space $(SL(T_xM)/SO(T_xM))/_{\RR}$ where $SO(T_xM)$ are the linear transformations that are an isometry with respect to the Riemannian inner product on $T_xM$.}). We refer the reader to \cite[\S 13.2.2]{BRH} for a detailed explaination of how Theorem \ref{teo-invpple} follows from the results of \cite{AV} in the case of surfaces (which is the one we will use here).

\section{A diffusion of a diffeomorphism}\label{s.diffusion}

Consider a smooth function $\varphi:S^1 \to [0,1]$ with the property that $\varphi'(t)>0$ for $t \in (0,\frac{1}{2})$ and $\varphi'(t)<0$ for $t \in (1/2,1)$ where we identify $S^1=[0,1]/_{1\sim 0}$. 

We can choose families of diffeomorphisms $g_1^t, g_2^t, g_3^t, g_4^t \in \Diff^\infty_{vol}(\TT^2)$ as follows: 

\begin{itemize}
\item $g_1^t(x,y) = (x + t, y)$, 
\item $g_2^t(x,y)= (x, y+ t)$, 
\item $g_3^t(x,y) = (x + t\varphi(y), y)$ and,
\item $g_4^t (x,y) = (x , y + t\varphi(x))$.
\end{itemize}

Note that for $t=0$ all diffeomorphisms are the identity and that the families are continuous in $\Diff^\infty_{vol}(\TT^2)$. Here we are considering coordinates $(x,y) \in \TT^2 \cong \RR^2/_{\ZZ^2} \cong S^1 \times S^1$. 

The following will be used to show that a certain random walk has non zero Lyapunov exponents and later to show it is uniformly expanding:

\begin{prop}\label{l.noninv}
Let $\nu$ be a measure which is not mutually singular with $\mathrm{vol}$. Then, there are no $\nu$-measurable line bundles or conformal structures in $\TT^2$ invariant under $\hat g_1=g_1^\alpha, \hat g_2= g_2^\beta, \hat g_3= g_3^a,  \hat g_4=g_4^b$ if $\alpha, \beta \in \RR\setminus \QQ$ and $a,b >0$. 
\end{prop}

\begin{proof}
First assume that there is a measurable line bundle, i.e. $(x,y) \mapsto \Phi((x,y)) \in \PP(T_{(x,y)} \TT^2)$ a $\nu$-measurable function that we assume verifies that $D\hat g_i(\Phi((x,y))) = \Phi(\hat g_i((x,y))$ for $\nu$-almost every $(x,y) \in \TT^2$. Take $\nu_0$ to be the absolutely continuous part of $\nu$ (that is, $(\nu-\nu_0) \perp \mathrm{vol}$). We will show that $\nu_0=0$. 

Let us first show that if $\Phi$ is invariant under $\hat g_1$ and $\hat g_3$ then $\Phi$ must be the line field $(x,y) \mapsto \RR  \begin{pmatrix} 1 \\ 0 \end{pmatrix} $ up to $\nu_0$-measure $0$. A symmetric argument using $\hat g_2$ and $\hat g_4$ says that $\Phi$ must be $(x,y) \mapsto \RR \begin{pmatrix} 0 \\ 1 \end{pmatrix}$ up to $\nu_0$-measure $0$, which implies that $\nu_0=0$. Here we are using coordinates $v=(x,y)$ on $\TT^2$ seen as $\RR^2 /_{\ZZ^2}$ and identifying $T_v\TT^2$ with $\RR^2$ via the coordinates $(x,y)$ (e.g. $\begin{pmatrix} 1 \\ 0 \end{pmatrix}$ is the vector tangent to the curve $(x+t,y)$).

Note that if $x \neq \{0,1/2\}$ then we have that for every $y \in S^1$ and direction $\xi \in \PP(T_{(x,y)}\TT^2) \cong \PP( \RR^2)$ we have that $D_{(x,y)}\hat g_3^n \xi \to \RR \begin{pmatrix} 1 \\ 0 \end{pmatrix}$. 

For every $\eps>0$, take $K \subset \TT^2$ be a compact set with $\nu_0(K) \geq (1-\eps) \nu_0(\TT^2)$ where $\Phi$ is continuous. It follows from Poincar\'e recurrence that for almost every $(x,y) \in K$ we have that there is $n_j \to \infty$ so that $\hat g_3^{n_j}(x,y) \to (x,y)$ and $\hat g_3^{n_j}(x,y) \in K$. By continuity and since $K \cap \{0,1/2\} \times S^1$ has measure zero, this implies that $\Phi((x,y)) = \RR \begin{pmatrix} 1 \\ 0 \end{pmatrix}$ for every $(x,y) \in K$. Since $\eps$ was arbitrary we deduce that $\Phi(x,y)$ is $\nu_0$-almost everywhere equal to $ \RR \begin{pmatrix} 1 \\ 0 \end{pmatrix}$. The same argument applied to $\hat g_2$ and $\hat g_4$ gives that $\Phi$ must be $\nu_0$-almost everywhere equal to $\RR \begin{pmatrix} 0 \\ 1 \end{pmatrix}$. Since these two full $\nu_0$-measure sets are disjoint, this implies that $\nu_0=0$. % This implies that the $\nu_0$-measure of a set where the bundle is invariant under all four diffeomorphisms has to be zero, and therefore we deduce that $\nu_0=0$ as we wanted to show. 

To conclude it is enough to show that there are no $\nu$-measurable conformal invariant structures. But this also follows from the fact that almost everywhere the norm of $D_{(x,y)}\hat g_3^n$ is unbounded and if we pick a compact set where the conformal structure is continuous, the same argument as above implies that this set must have zero measure under $\mathrm{vol}$. 
\end{proof}

The fact that $\mathrm{vol}$-plays a special role has to do with the kind of random walk we will chose. 

Fix $\cU$ an open set in $\Diff^\infty_{vol}(\TT^2)$. For small $\eps>0$ and $f_0 \in \cU$ such that $g_i^t \circ f_0 \in \cU$ for all $|t|\leq \eps$ and positive numbers $p_i$ so that $\sum_{i=0}^4 p_i=1$, we will consider $\hat \mu$ to be the following measure on $\Diff^{\infty}_{vol}(\TT^2)$ (supported on $\cU$):

\begin{equation}\label{eq:hatmu}
\hat \mu = p_0 \delta_{f_0} + \sum_{i=1}^4 p_i \int_{-\eps}^\eps \delta_{g_i^t \circ f_0} dt. 
\end{equation}

We call $\hat \mu$ a \emph{diffusion} of $f_0$. Note that $\hat \mu$ is very close to $\delta_{f_0}$ both in support and in the weak-$\ast$-topology as we take $\eps\to 0$ (and independently of the values of $p_i$). 

\begin{prop}\label{p.acstationary}
If $\nu$ is a $\hat \mu$-stationary measure then it is not mutually singular with respect to $\mathrm{vol}$.  
\end{prop}

\begin{proof}
Since $\hat \mu^{(k)} \ast \nu = \nu$, it is enough to show that for every probability measure $\eta$ in $\TT^2$ we have that $\hat \mu^{(2)} \ast \eta$ has an absolutely continuous part with respect to $\mathrm{vol}$. We can write: 

$$ \eta (E) = \int_{E} \delta_{v}(E) d\eta(v). $$

We define $\hat \eta = \hat \mu^{(2)} \ast \eta$ and we get that: 

$$\hat \eta(E)= \hat \mu^{(2)} \ast \eta(E) = \int_{\cU} \int_{E} \delta_{f(v)} (E) d\eta(v) d\hat \mu^{(2)}(f) . $$

Exchanging the order of integration we can compute, for some $v \in \TT^2$ the value of the measure in $\TT^2$ defined by

$$ \hat \eta_v:= \int_{\cU} \delta_{f(v)} d\hat\mu^{(2)}(f),  $$

And we get that $\hat \eta (E) = \int \hat \eta_v (E) d\eta(v)$ for every $E \subset \TT^2$ measurable. 

Write $d\hat \eta_v = \rho_v d\mathrm{vol} + \hat \eta_v^\perp$ where $\rho_v$ is a $L^1$ density and $\hat \eta_v^\perp$ is mutually singular with respect to $\mathrm{vol}$.  We claim that there is $c_0>0$ independent of $\eta$ such that $\int \rho_v d\mathrm{vol} > c_0$ for every $v \in \TT^2$. This is because there is $c_0>0$ such that: 

$$ \hat \mu^{(2)} = c_0 \int_{-\eps}^\eps \int_{-\eps}^\eps \delta_{R_{(t,s)} \circ f_0} dtds + \tilde \mu. $$

\noindent where $R_{(t,s)}(x,y) = (x+t,y+s) \mathrm{mod} \ZZ^2$ and $\tilde \mu$ is a positive measure in $\TT^2$. In particular, this implies that for every $v \in \TT^2$ we have that 
$$\hat \eta_v = c_0 \int_{-\eps}^\eps \int_{-\eps}^\eps \delta_{R_{(t,s)} \circ f_0(v)} dtds + \int_{\cU} \delta_{f(v)} d\tilde \mu(v).$$ 

This implies that $\hat \eta$ also has an absolutely continuous part obtained by integrating $\rho_v$ against $\eta$. 
\end{proof}

\begin{remark}
In fact, one can use this argument to show that $\nu$ has to be absolutely continuous with respect to $\mathrm{vol}$ since one can see that each time one convolutes with $\hat \mu$ one gets more regularity. With some more work, one may show that it is in fact $\mathrm{vol}$, however we will not pursue this line since we will get it a posteriori appealing to \cite[Theorem C and D]{Chung}. See also \cite[Lemma 5]{BXY} for a similar argument.
\end{remark}

\section{Discretizing the diffusion and proof of Theorem \ref{teo-main}}\label{s.teo} 

We first show that the measure $\hat \mu$ defined in \eqref{eq:hatmu} is uniformly expanding: 

\begin{prop}\label{p.unifexp}
The measure $\hat \mu$ is uniformly expanding.
\end{prop}
\begin{proof}
Let us first show that if $\nu$ is an ergodic stationary measure then it has to be hyperbolic. If it is not, then we can apply Theorem \ref{teo-invpple} and deduce that $\nu$ is $f_i^t$ invariant for all $i$ where $f_i^t = g_i^t \circ f_0$. But this means that $\nu = (g^t_i)_\ast (f_0)_\ast \nu = (g_i^t)_\ast \nu$ for almost every\footnote{Since preserving a measure is a closed condition, we can actually say that $f_i^t$ and $g_i^t$ preserve $\nu$ \emph{for every} $t \in (-\eps,\eps)$, but this is not necessary.} $t \in (-\eps,\eps)$; where the last equality follows from the fact that $\nu$ is $f_0$-invariant. Since $g_2^\alpha \circ g_1^\beta$ is uniquely ergodic for some small irrational $\alpha,\beta$ we deduce that $\nu= \mathrm{vol}$. 

We first show that $\mathrm{vol}$ is a hyperbolic measure for the random walk. Using Theorem \ref{teo-invpple} it is enough to show that there are no $\mathrm{vol}$-measurable invariant line fields or conformal structures. But this follows from Proposition \ref{l.noninv} because if $E$ is (say) an invariant line field by $\hat \mu$-ae. $f$ it follows that it is $Df_0$ invariant as well as invariant under some $D(g_i^t \circ f_0)$ for almost all $t \in (-\eps,\eps)$. In particular, we deduce that there are irrational numbers $\alpha, \beta \in (-\eps,\eps)$ and positive numbers $a,b \in (0,\eps)$ such that $E$ is invariant under $D(g_1^\alpha \circ f_0)$, $D(g_2^\beta \circ f_0)$, $D(g_3^a\circ f_0)$ and $D(g_4^b\circ f_0)$. Using that the line field is $Df_0$ invariant we deduce it has to be $Dg_1^\alpha, Dg_2^\beta, Dg_3^a, Dg_4^b$ invariant, which is impossible due to Proposition \ref{l.noninv}. The same argument applies for measurable conformal structures.   

This implies that every $\hat \mu$-ergodic stationary measure $\nu$ is hyperbolic. 

Now we want to show that the stable direction of $\nu$ is not non-random.  But this follows using the same argument, applying Proposition \ref{p.acstationary} that shows that $\nu$ has an absolutely continuous part and so one can apply Proposition \ref{l.noninv} to show that the stable direction cannot be non-random. 
\end{proof}

Now we are in conditions of showing:

\begin{proof}[Proof of Theorem \ref{teo-main}]
We have established that $\hat \mu$ is uniformly expanding. That is, there is $N>0$ such that for every $v \in \TT^2$ and $w \in T_v \TT^2$ unit vector we have that that 

\begin{equation}
\int \log \|D_v f w \| d\hat \mu^{(N)}(f) > 4. 
\end{equation}

Consider a sequence $\mu_n$ of finitely supported measures whose support is contained in the support of $\hat \mu$ such that $\mu_n \to \hat \mu$ in the weak star topology. It follows that for every $v \in \TT^2, w \in T_v \TT^2$  there is $n_0=n_0(v,w)$ we have that for $n> n_0$ we have 

\begin{equation}\label{eq:mun}
  \int \log \|D_v f w \| d \mu_n^{(N)}(f) > 3. 
  \end{equation}

Using the fact that the support of $\mu_n$ is contained in $\cU$ we know that there is $\delta>0$ such that if $d(v,v')<\delta$ and $d(w,w') < \delta$ then if $n>n_0(v,w)$, then:

\begin{equation}\label{eq:mun2}
  \int \log \|D_{v'} f w' \| d \mu_n^{(N)}(f) > 2. 
  \end{equation}
  
This implies that one can choose a uniform $n_1$ so that if $n>n_1$ then \eqref{eq:mun2} holds for every $v,w$ which is what we want to prove.   

\end{proof}

\section{Some comments and possible extensions}\label{s.comments}
\subsection{Equidistribution} 
Here we prove Corollary \ref{cor-main} and make some additional comments. 

\begin{proof}[Proof of Corollary \ref{cor-main}]
Using \cite[Theorem D]{Chung} we know that for a uniformly expanding measure, every orbit is either finite or dense, and moreover, dense orbits equidistribute. Since in Theorem \ref{teo-main} we have constructed a uniformly expanding measure with finite support in every open set of $\Diff^\infty_{vol}(\TT^2)$ we deduce that for every $\mu'$ in a neighborhood of $\mu$ in the weak-$\ast$ topology such that the support\footnote{This is to guarantee that $\mu'$ is also uniformly expanding.} of $\mu'$ is contained in a neighborhood of the one of $\mu$ we have that every orbit by the random walk generated by $\mu'$ is either finite or equidistributed with respect to $\mathrm{vol}$ (thus dense). 

A finite orbit must be invariant under $\mu'$-a.e. diffeomorphism of the support. Recall that without loss of generality, we can assume that $\mu$, the measure constructed in Theorem \ref{teo-main} is supported on finitely many diffeomorphisms including $f_0$ (a certain diffeomorphism chosen somewhere) and $g_1^\alpha \circ f_0$ with some $\alpha \notin \QQ$. This implies that if there is a finite orbit then it must be invariant under both diffeomorphisms, but this would imply that it is invariant under $g_1^\alpha$ which does not have finite orbits. This shows that every orbit of the random walk generated by $\mu$ equidistributes towards $\mathrm{vol}$.  

To get the stable ergodicity of the semigroup generated by the elements of the support of $\mu$ we just need to notice that after perturbation it is not possible that every point has finite orbit because every stationary measure is hyperbolic and therefore there must be infinite orbits\footnote{In fact, it gives stable ergodicity of the random walk, a concept we have not defined, see \cite{DK,Zhang}.}.
\end{proof}

In fact, one can easily show that in a neighborhood of $\mu$ there is a residual subset of measures that generate a minimal semigroup: 

\begin{prop}
Given open sets $\cU_1, \ldots, \cU_k \in \Diff^\infty_{vol}(\TT^2)$ there is a residual subset $\cR$ of the product $\cU_1 \times \ldots \times \cU_k$ such that if $(f_1, \ldots, f_k) \in \cR$ then the diffeomorphisms $f_1, \ldots, f_k$ do not have a common invariant finite set. 
\end{prop}
\begin{proof}
We can assume that $f_1$ is Kupka-Smale, so in particular, it has finitely many orbits of period $\leq N$ for every $N>0$. We claim that for every $N$, the set $\cA_N$ of $\cU_2 \times \ldots \times \cU_k$ consisting of diffeomorphisms $(f_2, \ldots, f_k)$ that do not preserve the set 
$$P_N = \{ x \in \TT^2 \ : \ \text{the orbit of } x \text{ under } f_1 \text{ has less than } N \text{ points }\},$$
\noindent is open and dense in $\cU_2 \times \ldots \times \cU_k$. It is clear that $\cA_N^c$ is closed, and since $P_N$ is a finite set for every $N$, a small perturbation of some of the diffeomorphisms allows one to remove the family from $\cA_N^c$ so we complete the proof. 
\end{proof}

\begin{remark}\label{r.nonvolumepres}
We point out that the condition of being uniformly expanding is open among measures supported in $\Diff^\infty(\TT^2)$ and not just those preserving volume. The results of \cite{BRH} hold in this more general setting, but instead of equidistribution to $\mathrm{vol}$ give equidistribution to some SRB-measure (this is enough to get robust transitivity of the semigroup, for instance). We refer the reader to that paper for more discussion. We also point out the recent preprint \cite{FNR} where a notion of stable ergodicity outside volume preserving semigroups is proposed. %, we believe that maybe the ideas here allow to obtain other examples of semigroups with this property. 
\end{remark}

\subsection{Bound on the cardinality of the support}
Theorem \ref{teo-main} can be compared with \cite[Theorem A]{Chung}. On the one hand, here we obtain perturbations of any map and show that they are uniformly expanding, but on the other hand the results in \cite{Chung} are quite deeper as they allow to control the number of diffeomorphisms in the support (the results are more \emph{effective}). In particular, the way that uniform expansion is checked in \cite{Chung} (see \cite[Proposition 5.4]{Chung}) allows to check it for a given family of diffeomorphisms while here we use some abstract criteria that produces uniform expansion for a continuous measure, and then gives the finite support by a discretisation argument that does not control the number of diffeomorphisms in the support. 

A particularly puzzling question that one can pose in the direction of trying to control the number of diffeomorphism in the support of a measure which is uniformly expanding close to a given diffeomorphism is the following:

\begin{quest}
Is it possible to show that for $\alpha, \beta \in \RR \setminus \QQ$ and small $a,b>0$ we have that the diffeomorphisms $g_1^\alpha, g_2^\beta, g_3^a, g_4^b$ do not leave any $\nu$-measurable line field where $\nu$ is any measure which is quasi-invariant under all the diffeomorphisms? (cf. Proposition \ref{l.noninv}.)
\end{quest}

We can show the following. Denote $\hat g_1 = g_1^\alpha$, $\hat g_2= g_2^\beta$,  $\hat g_3=g_3^a$, $\hat g_4= g_4^b$ for small $\alpha, \beta \in \RR\setminus \QQ$ and $a,b>0$.  

\begin{prop}
Let $\hat \mu$ be a (symmetric) probability measure in $\Diff^\infty_{vol}(\TT^2)$ such that $\hat \mu (\{\hat g_i^{\pm 1}\}_{i=1,2,3,4})=1$ and that $\hat \mu(\{\hat g_i\})=\hat \mu(\{ \hat g_i^{-1}\})>0$ for $i=1,2,3,4$. Then $\hat \mu$ is uniformly expanding.
\end{prop}

\begin{proof} 
Using Proposition \ref{l.noninv} we know that every stationary measure is hyperbolic (note that if $\nu$ is a stationary measure which is not hyperbolic, it must be $\hat g_i$-invariant for all $i$ by Theorem \ref{teo-invpple} and thus has to be $\mathrm{vol}$ and preserve a measurable line field or conformal structure and then one can apply Proposition \ref{l.noninv}). 

So, we only need to check that if $\nu$ is a $\hat \mu$-stationary measure, then the stable Oseledets direction $E^s$ (measurable and defined $\hat \mu^{\ZZ_{\geq 0}} \times \nu$-a.e.) is not invariant under all $g_i$. 

Assume by contradiction that $E^s$ is invariant. Consider the skew-product dynamics $F: \Sigma^+ \times \TT^2 \to \Sigma^+ \times \TT^2$ where $$\Sigma^+ = \{\hat g_1, \hat g_1^{-1}, \hat g_2, \hat g_2^{-1},\hat  g_3, \hat g_3^{-1}, \hat g_4, \hat g_4^{-1}\}^{\ZZ_{\geq 0}}$$ \noindent which we know leaves invariant the measure $\hat \nu=\hat \mu^{\ZZ_{\geq 0}} \times \nu$. Denote $\omega = (\omega_i)_{i\geq 0}$ a generic word in $\Sigma^+$. By assumption, the bundle $(\omega, (x,y)) \mapsto E^{s}(\omega, (x,y))$ associated with the stable Oseledets bundle of the hyperbolic measure $\hat \nu$ does not depend on $\omega$. Since $\hat \mu$ is symmetric we get if we define $S: \Sigma^+ \times \TT^2 \to \Sigma^+ \times \TT^2$ as $S(\omega, v) = (\hat \omega , v)$ where $\hat \omega = (\omega_0^{-1}, \omega_1^{-1}, \ldots )$ we get that $S_\ast \hat \nu = \hat \nu$. 

Then we can define the measurable function $\psi: \Sigma^+ \times \TT^2 \to \RR$ as $(\omega, (x,y)) \mapsto \log \|D_{(x,y)}\omega_0|_{E^s((x,y))}\|$, then, using that $S_\ast \hat \nu = \hat \nu$  and that $\psi \circ S = - \psi$ we get  
$$\int_{\Sigma^+ \times \TT^2} \psi d\hat \nu = \int_{\Sigma^+ \times \TT^2} \psi d S_\ast \hat \nu = \int_{\Sigma^+ \times \TT^2} \psi \circ S d\hat \nu= $$ $$=  -\int_{\Sigma^+ \times \TT^2} \psi d\hat \nu  \Rightarrow  \int_{\Sigma^+ \times \TT^2} \psi d\hat \nu = 0$$ Therefore, the Birkhoff ergodic theorem implies that $\hat \nu$ has a zero Lyapunov exponents, a contradiction. (Compare with \cite[Lemma 13.2]{BRH}.)  
\end{proof}

Using the same argument as in Proposition \ref{p.unifexp} to show that we can have uniformly expanding measures in any open set with support in a uniformly bounded number of diffeomorphisms one could use the diffeomorphisms of the previous proposition if one can answer the following which may have interest by itself (compare with \cite[\S 2]{Chung}, \cite[\S 4.3]{LX}): 

\begin{quest}
Assume that a  finitely supported probability measure $\hat \mu$ on $\Diff^{\infty}_{vol}(\TT^2)$ is uniformly expanding. Is it true that if $\nu$ is a measure which is quasi-invariant under every $f$ in $\mathrm{supp}(\hat \mu)$ it follows that there are no $\nu$-measurable line fields (defined $\nu$-ae) which are invariant under every $f \in \mathrm{supp}(\hat \mu)$?
\end{quest}

\subsection{Higher dimensions}
We comment now on the extensions to higher dimensions (or other surfaces). First of all, we note that in \cite{BEF} uniform expansion will be used to obtain rigidity statements about stationary measures, so the notion is still relevant in higher dimensions. Here we must remark that uniform expansion in higher dimensions admits different formulations. One could keep the exact Definition \ref{def-Nexp} which also makes sense in higher dimensions, but one could also ask for something stronger, in particular asking that the condition holds not only for vectors but also for exterior products up to codimension one. For such condition, one can also obtain a result analogous to Theorem \ref{teo.equivunifexp}.  

About the proofs, we only used $\TT^2$ to have specific coordinates and have a simple proof of Proposition \ref{l.noninv}. Let us sketch a way to obtain a similar result for general closed manifolds. 

Fix any closed manifold $M$. For every $x \in M$ we can find a continuous finite parameter familly $g_x^a \in \Diff^\infty_{vol}(M)$ with $a \in (-1,1)^\ell$ so that: 

\begin{itemize}
\item There is a neighbourhood $U_x$ of $x$ so that for every $y \in U$ we have that the map $a \mapsto g_x^a(y)$ is a smooth map from $(-1,1)^\ell$ to $M$ and the derivative at $0$ is surjective. 
\item For every $1 \leq i \leq d-1$ and $w \in \mathrm{Gr}_i(T_xM)$ the map $a \mapsto D_x g_x^a(w)$ is a smooth map from $(-1, 1)^\ell$ to $\mathrm{Gr}_i(TM)$ whose derivative is surjective at $0$. 
\end{itemize}

Note that such a family of maps can be constructed first in $\RR^d$ with respect to the origin\footnote{One uses the first $d$-parameters to move the origin and get the first condition, and then, for each $1 \leq i \leq d-1$ one can use $\mathrm{dim}(\mathrm{Gr_i}(\RR^d))$ parameters to fulfil the second condition for each $i$.} and then send them to $M$ via coordinate charts. We can find a finite set $x_1, \ldots, x_n$ of $M$ so that the neighborhoods $U_{x_i}$ of the first item cover $M$. 

By considering any diffeomorphism $f_0 \in \Diff^\infty_{vol}(M)$ we can construct a diffusion $\hat \mu$ by considering a small delta at $f_0$ together with measures that charge uniformly the submanifolds $a \mapsto g_{x_i}^a$ for $a \in (-\eps,\eps)^\ell$ in such a way that $\hat \mu$ is supported in a given neighborhood $\cU$ of $f_0$ in $\Diff^\infty_{vol}(M)$. Proposition \ref{p.acstationary} will work exactly the same in this context to give that any stationary measure will have some absolutely continuous part. It is not hard to see that for such a family of diffeomorphisms, the unique common invariant measure has to be $\mathrm{vol}$ so we deduce that if there is a stationary measure which is not expanding, then it must be $\mathrm{vol}$. Finally, using the second property and Theorem \ref{teo-invpple} we can see that $\mathrm{vol}$ is also expanding. 

Using the fact that stationary measures have an absolutely continuous part, and that we can take points everywhere using the diffeomorphisms that we chose, we can argue as in 
Proposition \ref{p.unifexp} (changing Proposition \ref{l.noninv} by a finer use of the second property defining the parametric families of diffeomorphisms) to deduce that $\hat \mu$ cannot leave invariant any bundle and therefore it is uniformly expanding. Finally, a discretization using the openness of the uniform expansion property allows to make the support finite and show the analogous result as Theorem \ref{teo-main} for any closed manifold. 

\medskip

{\small \emph{Acknowledgements:} This work reports some discussions with Alex Eskin in Fall 2019 at the Institute for Advanced Study, I'd like to thank him for his patience and generosity in explaining me the objects described here. The work benefited from discussions with Sylvain Crovisier, Davi Obata, Mauricio Poletti and Zhiyuan Zhang as well as  comments and suggestions on the text by Sylvain Crovisier and Davi Obata.}%%%%%%%%%%%%%%%%%%%%%%%%%%%%

\end{document}